\documentclass[12pt]{amsart}
\usepackage[latin1]{inputenc}
\usepackage{mathptmx}
\usepackage[norelsize]{algorithm2e} 
\usepackage{algpseudocode}
\usepackage{amscd}
\usepackage{color}

\theoremstyle{theorem}
\newtheorem{theorem}{Theorem}

\newtheorem{corollary}[theorem]{Corollary}

\newtheorem{proposition}[theorem]{Proposition}

\newtheorem{problem}[theorem]{Problem}

\theoremstyle{definition}
\newtheorem{remark}[theorem]{Remark}
\newtheorem{definition}[theorem]{Definition}
\newtheorem{example}[theorem]{Example}

\let\Bbb=\mathbb
\let\phi=\varphi
\def\NZQ{\Bbb}
\def\NN{{\NZQ N}}

\def\ZZ{{\NZQ Z}}

\def\Mon{\operatorname{Mon}}

\def\Ann{\operatorname{Ann}}
\def\hdepth{\operatorname{hdepth}}
\def\sdepth{\operatorname{sdepth}}

\def\Min{\operatorname{Min}}

\let\oldbigwedge\bigwedge
\def\BIGwedge{{\textstyle\oldbigwedge}}
\def\medwedge{{\scriptstyle\oldbigwedge}}
\def\bigwedge{\mathchoice{\BIGwedge}{\BIGwedge}{\medwedge}{}}

\let\iso=\cong

\let\epsilon=\varepsilon
\let\tilde=\widetilde

\begin{document}


\title{An algorithm for computing the multigraded Hilbert depth of a module}

\author{Bogdan Ichim and Andrei Zarojanu}

\address{Simion Stoilow Institute of Mathematics of the Romanian Academy, Research Unit 5, P.O. Box 1-764, 014700 Bucharest, Romania}
\email{bogdan.ichim@imar.ro}

\address{Simion Stoilow Institute of Mathematics of the Romanian Academy, Research Unit 5, P.O. Box 1-764, 014700 Bucharest, Romania
and Faculty of Mathematics and Computer Sciences, University of Bucharest, Bucharest, Romania}
\email{andrei\_zarojanu@yahoo.com}


\begin{abstract}
A method for computing the multigraded Hilbert depth of a module was presented in \cite{IJ}.
In this paper we improve the method and
we introduce an effective algorithm for performing the computations. In a particular case, the algorithm may also be easily adapted for computing  the Stanley depth of the module. We further present interesting examples which were
found with the help of an experimental implementation of the algorithm \cite{IZ}. Thus, we completely solve several open problems proposed by Herzog in \cite{H}.
\end{abstract}

\maketitle

\section{Introduction}

In this paper we introduce an algorithm for computing the Hilbert depth of a finitely generated multigraded module $M$ over the standard multigraded polynomial ring $R=K[X_1,\dots,X_n]$. The algorithm is based on the method
presented in \cite{IJ} and some extra improvements. It may also be adapted for computing the Stanley depth of $M$ if $\dim_K M_a \le 1$ for all $a\in \ZZ^n$. Further, we provide an experimental implementation of the algorithm
\cite{IZ} in CoCoA \cite{C} and we use it to find interesting examples. As a consequence, we give complete answers to the following open problems proposed by Herzog in \cite{H}:

\begin{problem}\label{P1:Herzog}\cite[Problem 1.66]{H} Find an algorithm to compute the Stanley depth for finitely
generated multigraded $R$-modules $M$ with $\dim_K M_a \le 1$ for all $a\in \ZZ^n$.
\end{problem}

\begin{problem}\label{P2:Herzog}\cite[Problem 1.67]{H} Let $M$ and $N$ be finitely generated multigraded $R$-modules. Then
$$ \sdepth(M \oplus N) \geq \Min \{ \sdepth(M),\sdepth(N)\}.$$
Do we have equality?
\end{problem}

\begin{problem}\label{P3:Herzog}\cite[Text following Problem 1.67]{H} In the particular case that $I\subset R$ is a monomial ideal, does
$\sdepth(R \oplus I)=\sdepth I$ hold?
\end{problem}

The examples are contained in Section \ref{Experiments}. One may read and check them directly (it is enough to see that each square--free monomial of the given modules is present
one and only one time in the given decomposition). The reader interested only in the answers to Problems \ref{P2:Herzog} and \ref{P3:Herzog} may  skip the rest of the paper and jump directly to Section \ref{Experiments}.

In recent years, \emph{Stanley decompositions} of multigraded modules
over $R$ have been discussed
intensively. These decompositions, introduced by Stanley in \cite{S},
break the module $M$ into a direct sum of \emph{Stanley spaces},
each being of type $mS$ where $m$ is a homogeneous element of $M$,
$S=K[X_{i_1},\dots,X_{i_d}]$ is a polynomial subalgebra of $R$ and $S\bigcap \Ann m=0$. One says that $M$ has \emph{Stanley depth} $s$, $\sdepth M=s$,  if
one can find a Stanley decomposition in which $d\ge s$ for each
polynomial subalgebra involved, but none with $s$ replaced by $s+1$.
\medskip

The computation of the Stanley depth is not an easy task, due mainly to its combinatorial  nature. A first step was done by Herzog, Vladoiu and Zheng in \cite{HVZ}, where they introduced a method for computing the Stanley depth of
a factor of a monomial ideal which was later developed into an effective algorithm by Rinaldo in \cite{Ri}. Some remarkable results in the study of the Stanley depth in the multigraded case were also presented by Apel (see
\cite{A1}, \cite{A2}), Herzog et al.\ (see \cite{HP}, \cite{HSY}) and Popescu et al.\ (see \cite{AP}, \cite{P}).
\medskip

Hilbert series are the most important numerical invariants of finitely generated graded and multigraded modules over $R$ and they form the bridge from commutative algebra to its combinatorial applications (we refer here to
classical results of Hilbert, Serre, Ehrhart and Stanley, see \cite{BG}). A new type of decompositions for multigraded modules $M$ depending only on the Hilbert series of $M$ was introduced by Bruns, Uliczka and Krattenthaler in
\cite{BKU} and called \emph{Hilbert decompositions}. They are a weaker type of decompositions not requiring
the summands to be submodules of $M$, but only vector subspaces isomorphic to polynomial subrings. The notion of \emph{Hilbert depth} $\hdepth M$ is defined accordingly. Several results concerning both the graded and multigraded
cases were presented in \cite{BKU2}, \cite{JU} and \cite{U}. All of them are based on both combinatorial and algebraic  techniques. Algorithms for computing the graded Hilbert depth of a module were introduced first in \cite{p},
then in a more complex setup in \cite{BMU}, while a method for computing the multigraded Hilbert depth of a module was presented in \cite{IJ}.
\medskip

The paper is organized as follows. In Section \ref{Pre} we recall some results concerning Hilbert depth that will be used in this paper.

Section \ref{Improve} is devoted to improve the method presented in \cite{IJ} by restricting as much as possible the search for a suited Hilbert decomposition.
Theorem \ref{theo:new} shows the existence of upper--discrete Hilbert partitions of degree $s$ for $\hdepth M \geq s$. We conclude that for the effective computation of the Hilbert depth it is better to consider only this kind of
partitions. The result generalize both \cite[Lemma 3.4]{Ri} and \cite[Lemma 3.3]{Sh} (notice that, in the particular case of a factor of a monomial ideal, the Hilbert partitions coincide with the poset partitions considered by
Rinaldo and Shen).

In Section \ref{Algo} we introduce a \emph{recursive} algorithm for computing the multigraded Hilbert depth of a module (see Algorithm \ref{algo:hdepth}). The algorithm is relative easy to implement because of its recursive form
and may also be used directly for computing the Stanley depth in the case of a factor of a monomial ideal. A \emph{non-recursive} algorithm for computing Stanley depth in the case of a factor of a monomial ideal  was introduced in
\cite[Algorithm 1]{Ri}. For computing the Stanley depth in the case of a factor of a monomial ideal the computation times of the two algorithms are similar (comparing our implementation with the original implementation of \cite{Ri}, see Section \ref{Experiments}).

Hilbert decompositions are intimately related to Stanley decompositions: All Stanley decompositions are Hilbert decompositions; moreover, the latter are prerequisites to the existence of Stanley decompositions. In Section
\ref{AlgoS} we assume that $\dim_K M_a \le 1$ for all $a\in \ZZ^n$ and we show that Algorithm \ref{algo:hdepth} may be easily modified for computing Stanley depth in this case (see Algorithm \ref{algo:sdepth}). This solves
completely Problem \ref{P1:Herzog}.

In Section \ref{Experiments} we present the result of several computations done with the algorithm introduced in Section \ref{Algo}. We have experimented with an implementation of the algorithm in CoCoA and we have found an example
in dimension 4 which shows that the answer to Problem \ref{P2:Herzog} is \emph{No}, then an example in dimension 6 which shows that even the answer to the more particular case considered in Problem \ref{P3:Herzog} is \emph{No}.
A nice theorem of J. Uliczka \cite{U} arranged in a quick algorithm by A. Popescu in \cite{p} for computing the graded Hilbert depth and
several computations from \cite{p} has been the basis of the
search for these examples (remark that, in general, the graded Hilbert depth is bigger than the multigraded Hilbert depth, so we were
lucky to find these examples in relatively low dimension).
\medskip

We end this section with a vague remark. In the particular case of a normal affine monoid, suited Hilbert decompositions have already been used with success in order to design arguable the fastest available algorithms for computing
Hilbert series (see \cite{BI}, \cite{BIS} and \cite{BK}). It is an interesting open problem if it is possible to use suited Hilbert decompositions in order to design efficient algorithms for computing Hilbert series in other
cases.

\section{Prerequisites}\label{Pre}

Let $R=K[X_1,...,X_n],$ with $K$ a field, and let $M$ be a finitely generated $\ZZ^n$-graded $R$-module. In \cite{IJ} the authors presented a method for computing the multigraded Hilbert depth of $M$ by considering Hilbert
partitions of its Hilbert series. We refer the reader also to \cite{BKU} and \cite{HVZ} on which \cite{IJ} is based. In this section we recall the method of \cite{IJ}.

A natural partial order on $\ZZ^n$ is defined as follows: Given $a,b\in \ZZ^n$, we say that $a \preceq b$ if and only if $a_i\le b_i$ for $i=1,\ldots,n$. Note that $\ZZ^n$ with this partial order is a distributive lattice with meet
$a\wedge b$ and join $a\vee b$ being the componentwise minimum and maximum, respectively.
We set the interval between $a$ and $b$ to be
\[
[a,b]=\{c \in \ZZ^n \ |\  a \preceq c \preceq b \}.
\]

We first recall a definition and a result of Ezra Miller (see \cite{M}) that will be useful in the sequel.
Let $g\in \NN^n$. The module $M$ is said to be $\NN^n$-graded if $M_a=0$ for $a \notin \NN ^n$ and $M$ is said to be {\it positively $g$-determined} if it is $\NN ^n$-graded and the multiplication map $\cdot X_i : M_{a}
\longrightarrow M_{a+e_i}$ is an isomorphism whenever $a_i \ge g_i$.
A characterization of positively $g$-determined modules is given by the following.

\begin{proposition}\label{prop:ezra}\cite[Proposition 2.5]{M}
The module $M$ is positively $g$-determined if and only if the multigraded Betti numbers of $M$ satisfy $\beta_{0,a}=\beta_{1,a}=0$ unless $0 \preceq a \preceq g$.
\end{proposition}

Let
\[
\bigoplus_{a\in \ZZ^n} R(-a)^{\beta_{1,a}} \longrightarrow \bigoplus_{a\in \ZZ^n} R(-a)^{\beta_{0,a}} \longrightarrow M \longrightarrow 0,
\]
be a minimal multigraded free presentation of $M$ and assume for simplicity, and without loss of generality, that all $\beta_{0,a}=0$  (and \emph{a fortiori}
all $\beta_{1,a}=0$) if $a \notin \NN^n$.
\medskip

Let $g\in \NN^n$ be such that the multigraded Betti numbers of $M$ satisfy the equalities $\beta_{0,a}=\beta_{1,a}=0$ unless $0\preceq a \preceq g$. Then, according to Proposition \ref{prop:ezra}, the module $M$ is positively
$g$-determined. Let
\[
H_M(X)=\sum_{a \in \NN^n} H(M,a)X^a
\]
be the \emph{Hilbert series} of $M$ and consider the polynomial
\[
H_M(X)_{\preceq g}:=\sum_{\substack{0\preceq  a \preceq g}} H(M,a)X^a.
\]

For $a,b \in \ZZ^n$ such that $a\preceq b$, we set
\[
Q[a,b](X):=\sum_{a \preceq c \preceq b } X^{c}
\]
and call it the \emph{polynomial induced by the interval} $[a,b]$.

\begin{definition}\label{defi:Hpartition}
We define a \emph{Hilbert partition}  of the polynomial $H_M(X)_{\preceq g}$ to be an expression
\[
\mathfrak{P}: H_M(X)_{\preceq g}=\sum_{i \in I} Q[a^i,b^i](X)
\]
as a finite sum of polynomials induced by the intervals $[a^i,b^i]$.
\end{definition}

Further, we need the following notations. For  $a\preceq g$ we set $Z_{a}=\{X_j \ |\  a_j=g_j\}$. Moreover, we denote by $K[Z_{a}]$ the subalgebra generated by the subset of the indeterminates $Z_{a}$. We also define the map
\[
\rho:\{0\preceq a\preceq g\} \longrightarrow \NN, \quad \rho (a):=|Z_a|,
\]
and for $0 \preceq a \preceq b\preceq g$ we set
\[
\mathcal{G}[a,b]=\{c\in [a,b] \mid c_j=a_j \text{ for all } j \in \NN \text{ with } X_j \in Z_{b}\}.
\]
The main result of \cite{IJ} (which generalizes the main result of \cite{HVZ}) is:

\begin{theorem}\cite[Theorem 3.3]{IJ} \label{theo:main} The following statements hold:
\begin{enumerate}
\item Let $\mathfrak{P}:H_M(X)_{\preceq g}=\sum_{i=1}^r Q[a^i, b^i](X)$ be a Hilbert partition of $H_M(X)_{\preceq g}$. Then
\[
\mathfrak{D}(\mathfrak{P}): M\iso \bigoplus_{i=1}^r\Big(\bigoplus_{c\in \mathcal{G}[a^i,b^i]} K[Z_{b^i}](-c)\Big)  \eqno[\star]
\]
is a Hilbert decomposition of $M$. Moreover,
\[
\hdepth \mathfrak{D}(\mathfrak{P})=\min\{\rho(b^i):\  i=1,\ldots,r\}.
\]
\item Let $\mathfrak{D}$ be a Hilbert decomposition of $M$. Then there exists a Hilbert partition $\mathfrak{P}$ of $H_M(X)_{\preceq g}$ such that
\[
\hdepth \mathfrak{D}(\mathfrak{P})\ge \hdepth \mathfrak{D}.
\]
In particular, $\hdepth M$ can be computed as the maximum of the numbers $\hdepth \mathfrak{D}(\mathfrak{P})$, where $\mathfrak{P}$ runs over the finitely many Hilbert partitions of $H_M(X)_{\preceq g}$.
\end{enumerate}
\end{theorem}

We see that, in order to effectively compute the Hilbert depth of $M$, we may use the following corollary.

\begin{corollary}\cite[Corollary 3.4]{IJ}\label{cor:hdepth}
There exists a Hilbert partition
$$
\mathfrak{P}:H_M(X)_{\preceq g}=\sum_{i=1}^r Q[a^i, b^i](X)
$$
of $H_M(X)_{\preceq g}$ such that
$
\hdepth M =\min\{\rho(b^i):\  i=1,\ldots,r\}.
$
\end{corollary}

\section{Restricting the search for a good partition}\label{Improve}

As seen in the previous section, the Hilbert depth of $M$ can be computed by considering all Hilbert partitions of $H_M(X)_{\preceq g}$. In practice, the number of possible partitions can easily become huge.
For many practical purposes (for example, for implementation of the method in a computer program), one needs to restrict (as much as possible) the search for a partition which will finally provide the right Hilbert depth. In this
section, we show that an improvement is indeed possible. Our results are extending some of the ideas presented by Giancarlo Rinaldo in \cite{Ri} and Shen in \cite{Sh} for computations of Stanley depth in the case of a factor of a
monomial ideal to the general case of a finitely generated module.

Since many results in this section depend on a number $g\in \NN^n$ such that $M$ is positively $g$-determined, we shall assume that $g$ is fixed and known from previous computations (for example by using Proposition
\ref{prop:ezra}).

\begin{definition} Let $B$ be a subset of $\NN^n$ and $ 0 \leq s \leq n$. We define two subsets of $B$,
$$ B_{<s} := \{ a \in B : \rho(a) < s \} \quad  \text{and} \quad B_{\geq s} := \{ a \in B : \rho(a) \geq s \},$$
where $\rho$ is the function defined in Section \ref{Pre}.
\end{definition}

 Our purpose is to test whether $M$ has a partition $\mathfrak{P} $ whose hdepth is equal to $s$. To reach this goal set $B=\{a:\ X^a \text{ is a monomial of the polynomial } H_M(X)_{\preceq g}\} $ and consider $B$ as a disjoint union
 of the two sets defined above
$$ B = B_{<s} \cup B_{\geq s}.$$

It is easy to observe that if $\mathfrak{P}$ is a Hilbert partition of $H_M(X)_{\preceq g} $, then we may write $ \mathfrak{P} = A + A'$,
so that
$$ A = \sum\limits_{i \in I} Q[a^i,b^i](X), \ \ \ A' = \sum\limits_{j \in I'} Q[{a}^j,{b}^j](X)$$
where  $a^i \in B_{<s}$ and ${a}^j \in B_{\geq s}$ for all $i\in I$ and $j\in I'$. Then $\mathfrak{P}$ can further be refined to a new partition $\mathfrak{P}' = A + A''$
with
$$A''= \sum\limits_{j\in I''} Q[{a}^j,{a}^j](X)$$
where ${a}^j \in B_{\geq s}$ for all $j\in I''$.

Therefore, if a partition $\mathfrak{P} $ with $\hdepth=s$ exists,  then the part $A$ of $\mathfrak{P} $ is composed of polynomials induced by intervals $Q[a,b](X)$, where $a \in B_{<s}$ and $b \in B_{\geq s}$.
At first glance, in order to find $A$, we have  to consider for each element $ a \in B_{<s}$ all possible candidates $b \in B_{\geq s}$ with $a \preceq b$. In the following, we show that the list of candidates can be reduced
considerably.

\begin{proposition}\label{prop:1}
Let $P=Q[a,b](X)$ be a polynomial such that $b\preceq g$ and $\rho(a) < s \leq \rho(b)$. Then for each
$$ b^0 \in \Min \{ x\ :\  a \preceq x \preceq b ,\  \rho (x) \geq s \} $$
there exists a disjoint decomposition of $P$
$$ P = P_0 +  \sum\limits_{i=1}^{r} P_i,  \ \ \ \ \ \ \ \ \ \ \ \ \ \ (*) $$
such that $P_{0}$ is the polynomial induced by the interval $[a,b^0]$,  $P_i$ is the polynomial induced by the interval $[a^i,b^i]$, $b^r = b$ and $ \rho(b^i) \geq s$ for all $i=1, \ldots, r$.
\end{proposition}

\begin{proof}
We see that
\begin{align*}
P & =  (X_1 ^{a_1} + ... + X_1 ^{b_1})\cdots (X_n ^{a_n} + ... + X_n ^{b_n}) \\
  & = X^a (1 + ... + X_1 ^{b_1-a_1}) \cdots (1 + ... + X_n ^{b_n-a_n}),
\end{align*}
so we may assume for simplicity and without loss of generality that $a=(0,...,0) \in \NN^n$. Then we have
\begin{align*}
P & = (1 + X_1 + ... + X_1 ^{b_1}) \cdots (1 + X_n  + ... + X_n ^{b_n}) \\
  & = P_0 + \sum\limits_{i=1}^{r} P_i,
\end{align*}
where we set
$$P_0 = (1 + ... + X_1^{b^0_1}) \cdots (1+...+X_n ^ {b^0_n})$$ and
$$P_i = (1 + ... + X_1^{b^0_1}) \cdots (1 + ...+X_{i-1}^{b^0_{i-1}}) (X_i ^ {b^0_i + 1} + ... + X_i ^ {b_i}) (1 + ... + X_{i+1} ^{b_{i+1}}) \cdots (1 + ... + X_n ^{b_n})$$
for all $i=1,\ldots, r$ (in case $b^0_i=b_i$, the term $P_i$ is simply $0$). Thus $P_i$ is the polynomial
induced by the interval $[a^i, b^i]$, where $a^i=(0,\ldots,0,b^0_i+1,0,\ldots,0)$ and $b^i$ is given by
\begin{center}
\[ b^i_j = \left\{ \begin{array}{ll}
         b^0_j, & \mbox{if $j < i$},\\
        b_j, & \mbox{otherwise}.\end{array} \right. \]
\end{center}
Since $b^0 \preceq b^i \preceq b \preceq g$, we get that $\rho(b^i) \geq \rho(b^0) \geq s$, as needed.

We claim that $(*)$ is a partition of $[0,b]$. To prove this, it is enough to show $\Mon(P_i) \cap \Mon(P_j) \neq \emptyset$ if and only if $i=j$ and that the equality $ P = P_0 +  \sum\limits_{i=1}^{r} P_i$ holds.

For the equality, we will show that $\Mon(P) = \Mon(P_0) \cup \bigcup\limits_{i=1}^{r} \Mon(P_i)$. We have only to show that $ \Mon(P) \subset \Mon(P_0) \cup \bigcup\limits_{i=1}^{r} \Mon(P_i)$ because the other equality is
obvious.

Let $u \in \Mon(P),\ u = X^c$. If $c_1 > b^0_1$, then $ u \in \Mon(P_1)$, otherwise for sure $u \notin \Mon(P_1)$. If $c_1\le b^0_1$, we check whether $c_2 > b^0_2$. If so, then $ u \in \Mon(P_2)$, otherwise $u \notin \Mon(P_1) \cup
\Mon(P_2)$. So, after checking all the variables, we find that either (a): if $c_j \leq b^0_j$ for all $j=1,\ldots,i-1$ and $c_i > b^0_i$, then $u \in \Mon(P_i)$; or (b): if $c_i \leq b^0_i$ for all $i=1,\ldots,n$, then $u \in \Mon(P_0)$.
It is also clear from this description that $\Mon(P_i) \cap \Mon(P_j) \neq \emptyset$ if and only if $i=j$ .
\end{proof}

\begin{remark}\label{prop:2}In fact, in Proposition \ref{prop:1}, we have that $\rho(b^0)=s$. Indeed, we
may again assume that $a=(0,...,0) \in \NN^n$. Then, if $\rho(b^0)=t > s$, we may suppose  that $b^0 _i = g_i$ for all $i=1,\ldots,
t$.   We have $a<b'=(b^0 _1,...,b^0 _s,0,...,0) < b^0$, $\rho(b') = s$ and  we get a contradiction with the minimality of $b^0$.
\end{remark}

\begin{definition} Let $a\in B_{<s}$. We define the set
$$ B_{=s}(a) := \{ x \in B_{\ge s} : a\preceq x, \ \rho(x) = s \}.$$
\end{definition}

\begin{theorem}\label{theo:new}
Assume $ \hdepth M \geq s$. Then there exists a Hilbert partition
$$ \mathfrak{P} : H_M(X)_{\preceq g} = \sum\limits_{i=1}^{r} Q[a^i,b^i](X) $$
such that if $\rho(a^i) < s$, then $b^i\in B_{=s}(a)$.
\end{theorem}

\begin{proof}
Since $\hdepth M \geq s$, we have a partition on $H_M(X)_{\preceq g}$,
$$ \mathfrak{P} : H_M(X)_{\preceq g} = \sum\limits_{i=1}^{r} Q[a^i,b^i](X), $$
with $\rho(b^i) \geq s.$ If there exists $a^j$ such that $\rho(a^j) < s$ and $b^j$ is not minimal, we apply  Proposition \ref{prop:1} to the polynomial induced by the interval $[a^j,b^j]$ and use Remark \ref{prop:2} to complete the proof.
\end{proof}

\begin{example}\label{example1}
Let $R=K[X_1,X_2]$ with $\deg (X_1)=(1,0)$ and $\deg (X_2)=(0,1)$. Let $M=R \oplus (X_1,X_2)R$. Then we may choose $g=(1,1)$ and
\[
H_M(X_1,X_2)_{\preceq (1,1)}=1+2X_1+2X_2+2X_1X_2.
\]
In order to use Corollary \ref{cor:hdepth}  to get that $\hdepth M\ge 1$ (for details see  \cite[Example 3.5]{IJ}), one has to compute a full Hilbert partition,  for example the following
\begin{align*}
\mathfrak{P}_1:& (1+X_1+X_2+X_1X_2)+(X_1+X_1X_2)+X_2.\\
\end{align*}
In this case $s=1$, so we have that $B_{<1}=\{(0,0)\}$ and $ B_{=1}((0,0)) = \{(1,0),(0,1)\}$. By Theorem \ref{theo:new} we only have to cover $(0,0)$ with an interval ending in an element of $ B_{=1}((0,0))$. The
computation is simply reduced at obtaining one of the following two possible covers:
$$
\mathfrak{C}_1: (1+X_1),\ \  \ \ \mathfrak{C}_2: (1+X_2).
$$
\end{example}

\section{An algorithm for computing the multigraded Hilbert depth of a module}\label{Algo}

In this section we describe a recursive algorithm for computing the multigraded Hilbert depth of a module. The algorithm is presented in the form of a function that will be called recursively, thus realizing a backtracking search
for a Hilbert partition of a given $\hdepth$. The algorithm may also be used directly for computing Stanley depth in the case of a factor of monomial ideal. See also \cite[Algorithm 1]{Ri}, for a non-recursive algorithm for
computing Stanley depth in the case of a factor of a monomial ideal.

\allowdisplaybreaks
\begin{algorithm}[H]\label{algo:hdepth}
\SetKwFunction{FindElementsToCover}{{\bf FindElementsToCover}}
\SetKwFunction{FindPossibleCovers}{{\bf FindPossibleCovers}}
\SetKwFunction{Beg}{begin}
\SetKwFunction{En}{end}
\SetKwFunction{size}{size}
\SetKwFunction{CheckHilbertDepth}{{\bf CheckHilbertDepth}}
\SetKwData{Boolean}{Boolean}
\SetKwData{Container}{Container}
\SetKwData{Polynomial}{Polynomial}
\caption{Function that checks if $\hdepth\ge s$ recursively}

\KwData{$g\in \NN^n$, $s \in \NN$ and a polynomial $P(X)=H_M(X)_{\preceq g}\in \NN[X_1,...,X_n]$}
\KwResult{{\it true} if $\hdepth M \geq s$}
\Boolean \CheckHilbertDepth{g,s,P}\;
\Begin{
\nl \lIf {$P\notin \NN[X_1,...,X_n]$}{\Return{false}\;}
\nl \Container $E=$\FindElementsToCover{g,s,P}\;
\nl \lIf {$\size{E}=0$}{\Return{true}\;}
\nl \Else{
\nl \For { i=\Beg{E} \KwTo i=\En{E} }{
\nl \Container $C[i]$:=\FindPossibleCovers{g,s,P,$E[i]$}\;
\nl \lIf {\size{$C[i]$}$=0$}{\Return{false}\;}
\nl \For { j=\Beg{$C[i]$} \KwTo j=\En{$C[i]$} }{
\nl \Polynomial $\tilde{P}(X)=P(X)-Q[E[i],C[i][j]](X)$\;
\nl \lIf{\CheckHilbertDepth{g,s,$\tilde{P}$}=true}{\Return{true}\;}
}
}
\nl \Return{false}\;
}
}
\end{algorithm}

At each call, the function {\bf CheckHilbertDepth} checks one interval of type $[a,b]$ to see if the polynomial induced by it may be part of a suited Hilbert partition. All possible intervals are checked in a backtracking search.  A node of the searching tree is represented by a polynomial $P$. Below we describe the key steps.
\begin{itemize}
	\item line 1. If the polynomial $P$ does not have natural numbers as coefficients (positive coefficients), then it is not a sum of polynomials induced by intervals and is not a node in the searching tree.
	\item line 2. In this step $B_{<s}$ is computed and stored in a container. The container should provide some basic access functions (for example, we want to query its size).
	\item line 3.  If $B_{<s}$ is empty, then we are done. We have reached a good leaf of the searching tree.
	\item line 4,5,8. We generate and investigate all the children  of the node $P$.
	\item line 5,6. In this loop, for each $a\in B_{<s}$, we compute the set $B_{=s}(a)$ (here we use Theorem \ref{theo:new}).
	\item line 7. If $B_{=s}(a)$ is empty, we are in a bad node, and we should go back to the previous node.
	\item line 9,10. The child $\tilde{P}$ is generated in line 9 and investigated in the recursive call at line 10.
	\item line 11. If we have reached this point, then our search in this node has failed, and we should go back to the previous node. If we are at the root, then  $\hdepth<s$.
\end{itemize}

We conclude this section with a remark on the functions {\bf FindElementsToCover} and {\bf FindPossibleCovers}. At each node, they should compute the sets $B_{<s}$ and $B_{=s}(a)$ for all $a\in B_{<s}$. For a practical
implementation of the algorithm, it is quite inefficient to compute them at each node. It is likely better to adjust them for the newly generated child $\tilde{P}$ and pass them down as input data for main recursive function {\bf
CheckHilbertDepth}.

\section{An algorithm for computing the Stanley depth in a special case}\label{AlgoS}

In this section, we further assume that $\dim_K M_a \le 1$ for all $a\in \ZZ^n$, and we modify Algorithm \ref{algo:hdepth} for computing the Stanley depth in this case. The algorithm checks supplementary whether the Hilbert partition computed by Algorithm \ref{algo:hdepth}  induces a Stanley decomposition.

\begin{algorithm}[H]\label{algo:sdepth}
\SetKwFunction{FindElementsToCover}{{\bf FindElementsToCover}}
\SetKwFunction{FindPossibleCovers}{{\bf FindPossibleCovers}}
\SetKwFunction{Beg}{begin}
\SetKwFunction{En}{end}
\SetKwFunction{size}{size}
\SetKwFunction{CheckStanleyDepth}{{\bf CheckStanleyDepth}}
\SetKwData{Boolean}{Boolean}
\SetKwData{Container}{Container}
\SetKwData{Polynomial}{Polynomial}
\caption{Function that checks if $\sdepth\ge s$ recursively}

\KwData{$g\in \NN^n$, $s \in \NN$ and a polynomial $P(X)=H_M(X)_{\preceq g}\in \NN[X_1,...,X_n]$}
\KwResult{{\it true} if $\sdepth M \geq s$}
\Boolean \CheckStanleyDepth{g,s,P}\;
\Begin{
\nl \lIf {$P\notin \NN[X_1,...,X_n]$}{\Return{false}\;}
    \Container $E=$\FindElementsToCover{g,s,P}\;
    \lIf {$\size{E}=0$}{\Return{true}\;}
    \Else{
    \For { i=\Beg{E} \KwTo i=\En{E} }{
    \Container $C[i]$:=\FindPossibleCovers{g,s,P,$E[i]$}\;
    \lIf {\size{$C[i]$}$=0$}{\Return{false}\;}
    \For { j=\Beg{$C[i]$} \KwTo j=\En{$C[i]$} }{
\nl \While {$a \in \mathcal{G}[E(i),C[i][j]]$ }{
\nl \lIf {$K[Z_{C[i][j]}] \cap \Ann M_{a} \neq 0$ }{\Return{false}\;}}
    \Polynomial $\tilde{P}(X)=P(X)-Q[E[i],C[i][j]](X)$\;
    \lIf{\CheckStanleyDepth{g,s,$\tilde{P}$}=true}{\Return{true}\;}
}
}
    \Return{false}\;
}
}
\end{algorithm}

\medskip

The only difference from Algorithm 1 appears at lines 2, 3. Here we check whether the Hilbert decomposition that we found is a Stanley decomposition. For this we use \cite[Proposition 4.4]{IJ}. The only thing to prove is that the conditions at lines 1, 3 ensure that $P$ is inducing a Stanley decomposition. Assume that for all $a \in \mathcal{G}[E(i),C[i][j]]$ we have that $K[Z_{C[i][j]}] \cap \Ann M_{a}= 0$. Let $0 \neq m_a\in M_{a}$. Since  $\dim_K M_{a} = 1$ we
have that $\Ann m_a=\Ann M_{a}$, so $K[Z_{C[i][j]}] \cap \Ann m_{a}= 0$. Then $m_a K[Z_{C[i][j]}]$ is a Stanley space. Finally, since all the coefficients of $P$ are $\le 1$, the condition at line 1 assures that they  do not
overlap.

We end with a vague remark. It is easy to see that for two intervals
$$
[a_i,b_i] \cap [a_j,b_j] \neq \emptyset \Longleftrightarrow a_i \vee a_j < b_i \wedge b_j.
$$
Since in this particular case the intervals do not overlap, for a practical implementation of the algorithm one may take advantage of this fact by saving the intervals and replacing the test needed at line 1.

\section{Computational results}\label{Experiments}

In this section, we present the results of our experiments with an implementation of the Algorithm \ref{algo:hdepth} in the computer algebra system CoCoA \cite{C}.
This implementation (as well as some test examples) is available online, see \cite{IZ}. The experiments were run on an Apple Mac Pro with a processor running at 3 Ghz.

\bigskip
Encouraged by the results obtained in \cite{p}, we have focused on obtaining a complete answer to Problems \ref{P2:Herzog} and \ref{P3:Herzog}.

The following example in dimension 4 shows that the answer to Problem \ref{P2:Herzog} is \emph{No}.

\begin{example}\label{directsum}
Let $n=4,\ M = R^2$ and $N=m$, where $m\subset R$ is the maximal ideal. It is known that $\Min \{ \sdepth(M),\sdepth(N)\} = 2$. The Hilbert partition $\mathfrak{P}_1$ presented below shows that
$ \hdepth(M \oplus N) = 3$.
\begin{align*}
\mathfrak{P}_1: & (1+X_1+X_2+X_3+X_1X_2+X_1X_3+X_2X_3+X_1X_2X_3)+\\
                & (1+X_1+X_2+X_4+X_1X_2+X_1X_4+X_2X_4+X_1X_2X_4)+\\
                & (X_1+X_1X_3+X_1X_4+X_1X_3X_4)+(X_2+X_1X_2+X_2X_3+X_1X_2X_3)+\\
                & (X_3+X_1X_3+X_3X_4+X_1X_3X_4)+(X_3+X_2X_3+X_3X_4+X_2X_3X_4)+\\
                & (X_4+X_1X_4+X_2X_4+X_1X_2X_4)+(X_4+X_2X_4+X_3X_4+X_2X_3X_4)+\\
                & \mbox{monomials of degree $\geq 3$}.
\end{align*}

The Hilbert partition $\mathfrak{P}_1$ induces a Hilbert decomposition, which in turn induces the Stanley decomposition
\begin{align*}
\overline{\mathfrak{D}(\mathfrak{P}_1)}: & (1,0,0)K[X_1,X_2,X_3] \oplus (0,1,0)K[X_1,X_2,X_4] \oplus\\
               & (0,0,X_1)K[X_1,X_3,X_4] \oplus (0,0,X_2)K[X_1,X_2,X_3] \oplus \\
               & (0,X_3,X_3)K[X_1,X_3,X_4] \oplus (0,X_3,0)K[X_2,X_3,X_4] \oplus \\
               & (X_4,0,X_4)K[X_1,X_2,X_4] \oplus (X_4,0,0)K[X_2,X_3,X_4] \oplus\\
               & (0,X_1X_2X_3,0)K[X_1,X_2,X_3] \oplus (X_1X_3X_4,0,0)K[X_1,X_3,X_4] \oplus \\
               & (0,0,X_1X_2X_4)K[X_1,X_2,X_4] \oplus (0,0,X_2X_3X_4)K[X_2,X_3,X_4] \oplus \\
               & (X_1X_2X_3X_4,0,0)K[X_1,X_2,X_3,X_4] \oplus (0,X_1X_2X_3X_4,0)K[X_1,X_2,X_3,X_4] \oplus\\ 	  	
               & (0,0,X_1X_2X_3X_4)K[X_1,X_2,X_3,X_4].
\end{align*}
It is clear that the multigraded Hilbert series of $M\oplus N$ coincide with the one of $\overline{\mathfrak{D}}(\mathfrak{P}_1)$.
That $\overline{\mathfrak{D}}(\mathfrak{P}_1)$ is indeed a Stanley decomposition follows once we have checked that the sums
\begin{align*}
               & (0,0,X_1)K[X_1,X_3,X_4] + (0,0,X_2)K[X_1,X_2,X_3] + \\
               & (0,X_3,X_3)K[X_1,X_3,X_4] + (0,X_3,0)K[X_2,X_3,X_4]\\
\end{align*}
and
\begin{align*}
               & (0,0,X_1)K[X_1,X_3,X_4] + (X_4,0,X_4)K[X_1,X_2,X_4] +\\
               & (X_4,0,0)K[X_2,X_3,X_4] + (0,0,X_1X_2X_4)K[X_1,X_2,X_4]
\end{align*}
are direct. It is easy to see that $ \sdepth(M \oplus N) \le 3$, since it is not a free module, or by using
the results of \cite{p}. We conclude that
$$
3=\sdepth(M \oplus N)=\hdepth(M \oplus N)>\Min \{ \sdepth(M),\sdepth(N)\}=2.
$$

\end{example}

\begin{remark} The computation time obtained with the experiment library depends on the input order of the coefficients of $H_M(X)_{\preceq g}$, and for each coefficient, on the order of the elements in the list of possible covers.
This is why we provide two implementations using different orders (see \cite{IZ}). The computation time for example \ref{directsum} with the implementation HdepthLib is 11805.446 seconds and with the implementation HdepthLib2 is
2760.213 seconds. Depending on the order, we also obtain different Hilbert partitions for example \ref{directsum}, an alternative Hilbert partition is the following:
\begin{align*}
\mathfrak{P'}_1: & (1+X_1+X_2+X_3+X_1X_2+X_1X_3+X_2X_3+X_1X_2X_3)+\\
                & (1+X_1+X_2+X_3+X_1X_2+X_1X_3+X_2X_3+X_1X_2X_3)+\\
                & (X_1+X_1X_2+X_1X_4+X_1X_2X_4)+(X_2+X_2X_3+X_2X_4+X_2X_3X_4)+\\
                & (X_3+X_1X_3+X_3X_4+X_1X_3X_4)+(X_4+X_2X_4+X_3X_4+X_2X_3X_4)+\\
                & (X_4+X_1X_4+X_2X_4+X_1X_2X_4)+(X_4+X_1X_4+X_3X_4+X_1X_3X_4)+\\
                & \mbox{monomials of degree $\geq 3$}.
\end{align*}
\end{remark}

\bigskip
The following example in dimension 6 shows that the answer to Problem \ref{P3:Herzog} is \emph{No}.

\begin{example} Consider $n=6$ and $I=m$,  where $m\subset R$ is the maximal ideal. It is known that $\sdepth(I)=\hdepth(I) = 3$ and we show that $ \sdepth(R \oplus I)=\hdepth(R \oplus I)= 4$. Remark that, while it is not as easy
to see as above, we have
$$
\sdepth(R \oplus I)\le \hdepth(R \oplus I)\le\hdepth_1(R \oplus I)= 4
$$
by \cite{p} (where $\hdepth_1(R \oplus I)$ is the standard graded Hilbert depth). The Hilbert partition $\mathfrak{P}_2$ presented below shows that $\hdepth(R \oplus I) = 4$.
\begin{align*}
\mathfrak{P}_2:& (1+X_1+X_2+X_3+X_4+X_1X_2+X_1X_3+X_1X_4+X_2X_3+X_2X_4+X_3X_4+X_1X_2X_3+\\
               & X_1X_2X_4+X_1X_3X_4+X_2X_3X_4+X_1X_2X_3X_4)+\\
               & (X_1+X_1X_2+X_1X_5+X_1X_6+X_1X_2X_5+X_1X_2X_6+X_1X_5X_6+X_1X_2X_5X_6)+\\
               & (X_2+X_2X_3+X_2X_5+X_2X_6+X_2X_3X_5+X_2X_3X_6+X_2X_5X_6+X_2X_3X_5X_6)+\\
               & (X_3+X_1X_3+X_3X_4+X_3X_5+X_1X_3X_4+X_1X_3X_5+X_3X_4X_5+X_1X_3X_4X_5)+\\
               & (X_4+X_1X_4+X_2X_4+X_4X_6+X_1X_2X_4+X_1X_4X_6+X_2X_4X_6+X_1X_2X_4X_6)+\\
               & (X_5+X_1X_5+X_2X_5+X_4X_5+X_1X_2X_5+X_1X_4X_5+X_2X_4X_5+X_1X_2X_4X_5)+\\
               & (X_5+X_3X_5+X_4X_5+X_5X_6+X_3X_4X_5+X_3X_5X_6+X_4X_5X_6+X_3X_4X_5X_6)+\\
               & (X_6+X_1X_6+X_2X_6+X_3X_6+X_1X_2X_6+X_1X_3X_6+X_2X_3X_6+X_1X_2X_3X_6)+\\
               & (X_6+X_3X_6+X_4X_6+X_5X_6+X_3X_4X_6+X_3X_5X_6+X_4X_5X_6+X_3X_4X_5X_6)+\\
               & (X_1X_2X_3+X_1X_2X_3X_4)+(X_1X_3X_5+X_1X_3X_4X_5)+(X_1X_3X_6+X_1X_2X_3X_6)+\\
               & (X_1X_4X_6+X_1X_2X_4X_6)+(X_1X_4X_5+X_1X_2X_4X_5)+(X_1X_5X_6+X_1X_2X_5X_6)+\\
               & (X_2X_3X_4+X_2X_3X_4X_5)+(X_2X_3X_5+X_2X_3X_5X_6)+(X_2X_4X_5+X_2X_4X_5X_6)+\\
               & (X_2X_4X_6+X_2X_3X_4X_6)+(X_2X_5X_6+X_2X_4X_5X_6)+(X_3X_4X_6+X_1X_3X_4X_6)+\\
               & \mbox{monomials of degree $\geq 4$}.
\end{align*}

The Hilbert partition $\mathfrak{P}_2$ induces a Hilbert decomposition, which in turn induces the Stanley decomposition $\overline{\mathfrak{D}}(\mathfrak{P}_2)$:
\allowdisplaybreaks
\begin{align*}
               & (X_5,X_5)K[X_3,X_4,X_5,X_6] \oplus (X_6,X_6)K[X_1,X_2,X_3,X_6] \oplus\\
               & (1,0)K[X_1,X_2,X_3,X_4] \oplus (0,X_1X_2X_3)K[X_1,X_2,X_3,X_4] \oplus\\
               & (X_5,0)K[X_1,X_2,X_4,X_5] \oplus (0,X_1X_4X_5)K[X_1,X_2,X_4,X_5] \oplus\\
               & (X_6,0)K[X_3,X_4,X_5,X_6] \oplus (0,X_1X_3X_6)K[X_1,X_2,X_3,X_6] \oplus\\
               & (X_1X_5X_6,0)K[X_1,X_2,X_5,X_6] \oplus (0,X_1)K[X_1,X_2,X_5,X_6] \oplus\\
               & (X_2X_3X_5,0)K[X_2,X_3,X_5,X_6] \oplus (0,X_2)K[X_2,X_3,X_5,X_6] \oplus \\
               & (X_1X_3X_5,0)K[X_1,X_3,X_4,X_5] \oplus (0,X_3)K[X_1,X_3,X_4,X_5] \oplus\\
               & (X_1X_4X_6,0)K[X_1,X_2,X_4,X_6] \oplus (0,X_4)K[X_1,X_2,X_4,X_6] \oplus\\
               & (X_2X_5X_6,0)K[X_2,X_4,X_5,X_6] \oplus (0,X_2X_4X_5)K[X_2,X_4,X_5,X_6] \oplus\\
               & (X_1X_3X_4X_6,0)K[X_1,X_3,X_4,X_6] \oplus (0,X_3X_4X_6)K[X_1,X_3,X_4,X_6] \oplus\\
               & (X_2X_3X_4X_5,0)K[X_2,X_3,X_4,X_5] \oplus (0,X_2X_3X_4)K[X_2,X_3,X_4,X_5] \oplus\\
               & (X_2X_4X_6,0)K[X_2,X_3,X_4,X_6] \oplus (0,X_2X_3X_4X_6)K[X_2,X_3,X_4,X_6] \oplus\\
               & (X_1X_2X_3X_5,0)K[X_1,X_2,X_3,X_5] \oplus (0,X_1X_2X_3X_5)K[X_1,X_2,X_3,X_5] \oplus \\
               & (X_1X_3X_5X_6,0)K[X_1,X_3,X_5,X_6] \oplus (0,X_1X_3X_5X_6)K[X_1,X_3,X_5,X_6] \oplus\\
               & (X_1X_4X_5X_6,0)K[X_1,X_4,X_5,X_6] \oplus (0,X_1X_4X_5X_6)K[X_1,X_4,X_5,X_6] \oplus \\
               & (X_1X_2X_3X_4X_5,0)K[X_1,X_2,X_3,X_4,X_5] \oplus (0,X_1X_2X_3X_4X_5)K[X_1,X_2,X_3,X_4,X_5] \oplus\\
               & (X_1X_2X_3X_4X_6,0)K[X_1,X_2,X_3,X_4,X_6] \oplus (0,X_1X_2X_3X_4X_6)K[X_1,X_2,X_3,X_4,X_6] \oplus\\
               & (X_1X_2X_3X_5X_6,0)K[X_1,X_2,X_3,X_5,X_6] \oplus (0,X_1X_2X_3X_5X_6)K[X_1,X_2,X_3,X_5,X_6] \oplus\\
               & (X_1X_2X_4X_5X_6,0)K[X_1,X_2,X_4,X_5,X_6] \oplus (0,X_1X_2X_4X_5X_6)K[X_1,X_2,X_4,X_5,X_6] \oplus\\
               & (X_1X_3X_4X_5X_6,0)K[X_1,X_3,X_4,X_5,X_6] \oplus (0,X_1X_3X_4X_5X_6)K[X_1,X_3,X_4,X_5,X_6] \oplus\\
               & (X_2X_3X_4X_5X_6,0)K[X_2,X_3,X_4,X_5,X_6] \oplus (0,X_2X_3X_4X_5X_6)K[X_2,X_3,X_4,X_5,X_6] \oplus\\
               & (X_1X_2X_3X_4X_5X_6,0)K[X_1,X_2,X_3,X_4,X_5,X_6] \oplus (0,X_1X_2X_3X_4X_5X_6)K[X_1,X_2,X_3,X_4,X_5,X_6].
\end{align*}

It is clear that the multigraded Hilbert series of $R\oplus I$ coincide with the one of $\overline{\mathfrak{D}}(\mathfrak{P}_2)$. That $\overline{\mathfrak{D}}(\mathfrak{P}_2)$ is indeed a Stanley decomposition follows after checking that the sums
\begin{align*}
               & (X_5,X_5)K[X_3,X_4,X_5,X_6] + (X_5,0)K[X_1,X_2,X_4,X_5] + (0,X_3)K[X_1,X_3,X_4,X_5]\\
\end{align*}
and
\begin{align*}
               & (X_6,X_6)K[X_1,X_2,X_3,X_6] + (X_6,0)K[X_3,X_4,X_5,X_6] + (0,X_1)K[X_1,X_2,X_5,X_6] +\\
               &  (0,X_2)K[X_2,X_3,X_5,X_6] + (0,X_1X_3X_6)K[X_1,X_2,X_3,X_6]
\end{align*}
are direct. We conclude that
$$
4=\sdepth(R \oplus I)=\hdepth(R \oplus I)>\sdepth(I)=\hdepth(I)=3.
$$
\end{example}
\bigskip

Finally, we have compared the CoCoA library for computing the Hilbert depth of a module \cite{IZ} with the CoCoA library for computing the Stanley depth of an ideal or factor of an ideal
implemented by Rinaldo \cite{Ri}. As test example, we have chosen the maximal ideal $m$ (the same as in \cite{Ri}). It is known (see for example \cite{BKU}) that, if $\dim R=n$, then
$$\sdepth m =\hdepth m = \Big\lceil \frac{n}{2} \Big\rceil.$$
We conclude that, while the library for computing the Hilbert depth is somewhat faster, the times are of similar magnitude.

\begin{table}[h]
\begin{tabular}{|r|r|r|}\hline
\rule[-0.1ex]{0ex}{2.5ex} dim & Stanley depth library time & Hilbert depth library time  \\ \hline
\rule{0ex}{1.5ex}5    & 0.044 s     & 0.033 s  \\ \hline
\rule{0ex}{2.5ex}6    & 0.141 s & 0.09 s  \\ \hline
\rule{0ex}{2.5ex}7    & 0.6 s  & 0.363 s     \\ \hline
\rule{0ex}{2.5ex}8     & 2.1 s & 0.835 s    \\ \hline
\rule{0ex}{2.5ex}9     & 10.312 s & 5.985 s      \\ \hline
\rule{0ex}{2.5ex}10     & 37.924 s & 13.418 s  \\ \hline
\rule{0ex}{2.5ex}11     & 200.552 s & 152.772 s  \\ \hline
\rule{0ex}{2.5ex}12    & 758.455 s  & 307.714 s  \\ \hline

\end{tabular}
\vspace*{2ex} \caption{Computation times }\label{times1}
\end{table}

\section{Acknowledgements}

The authors would like to thank  Dorin Popescu and Marius Vladoiu for their useful comments.

The first author was partially supported  by project  PN-II-RU-TE-2012-3-0161 and the second author was partially supported by project  PN-II-ID-PCE-2011-3-1023, granted by the Romanian National Authority for Scientific Research,
CNCS - UEFISCDI, during the preparation of this work.

\end{document}